\documentclass[12pt]{amsart}
\usepackage{euscript,epsf,amssymb,xr}

\let\oldmarginpar\marginpar
\renewcommand\marginpar[1]
{\oldmarginpar{\tiny\bf \begin{flushleft} #1 \end{flushleft}}}

\input xy
\xyoption{all}

\newcommand{\la}{\langle}
\newcommand{\ra}{\rangle}

\newtheorem{theorem}{\bf Theorem}[section]
\newtheorem{lemma}[theorem]{\bf Lemma}

\renewcommand{\AA}{{\Bbb A}}

\newcommand{\CC}{{\Bbb C}}

\newcommand{\FF}{{\Bbb F}}

\newcommand{\NN}{{\Bbb N}}

\newcommand{\RR}{{\Bbb R}}

\newcommand{\ZZ}{{\Bbb Z}}

\newcommand{\ggreat}{>\kern-.7ex>}
\newcommand{\ssmall}{<\kern-.7ex<}
\newcommand{\qu}{/\kern-.7ex/}
\newcommand{\exh}{\to\kern-1.8ex\to}

\newcommand{\cC}{{\EuScript{C}}}
\newcommand{\dD}{{\EuScript{D}}}

\newcommand{\gG}{{\EuScript{G}}}

\newcommand{\GL}{\operatorname{GL}}

\newcommand{\Aut}{\operatorname{Aut}}

\newcommand{\Diff}{\operatorname{Diff}}

\newcommand{\stable}{\operatorname{st}}

\newcommand{\Id}{\operatorname{Id}}

\newcommand{\Ker}{\operatorname{Ker}}

\newcommand{\Spec}{\operatorname{Spec}}

\newcommand{\un}{\underline}

\newcommand{\vt}{\vartriangleleft}

\title[Manifolds without odd cohomology and almost fixed points]
{Almost fixed points of finite group actions on manifolds without odd cohomology}

\author{Ignasi Mundet i Riera}
\address{Departament d'\`Algebra i Geometria\\
Facultat de Matem\`atiques\\
Universitat de Barcelona\\
Gran Via de les Corts Catalanes 585\\
08007 Barcelona \\
Spain}
\email{ignasi.mundet@ub.edu}

\date{\today}

\subjclass[2010]{57S17,54H15}

\thanks{This work has been partially supported by the (Spanish) MEC Projects MTM2012-38122-C03-02
and MTM2015-65361-P}

\begin{document}

\maketitle

\begin{abstract}
If $X$ is a smooth manifold and $\gG$ is a subgroup of $\Diff(X)$
we say that $(X,\gG)$ has the almost fixed point property if there exists
a number $C$ such that for
any finite subgroup $G\leq\gG$ there is some $x\in X$ whose stabilizer
$G_x\leq G$ satisfies $[G:G_x]\leq C$.
We say that $X$ has no odd cohomology if its integral cohomology is torsion free
and supported in even degrees. We prove that if $X$ is compact and possibly
with boundary and has no odd cohomology then $(X,\Diff(X))$ has the almost fixed
point property. Combining this with a result of Petrie and Randall we conclude
that if $Z$ is a non necessarily compact
smooth real affine variety,
and $Z$ has no odd cohomology, then $(Z,\Aut(Z))$ has the almost fixed point
property, where $\Aut(Z)$ is the group of algebraic automorphisms of $Z$
lifting the identity on $\Spec\RR$.
\end{abstract}

\section{Introduction}

\subsection{Smooth actions}
Let $X$ be a smooth manifold, possibly with boundary, and let $\gG$
be a subgroup of $\Diff(X)$. We say that $(X,\gG)$ has the fixed point property if for any finite subgroup $G\leq \gG$ the fixed point set $X^G$ is nonempty.

For which manifolds $X$ does the pair $(X,\Diff(X))$ have the fixed point property? This is trivially the case for asymmetric manifolds i.e., manifolds which do not admit an effective action of any nontrivial finite group (there exist many examples of asymmetric manifolds, see e.g. \cite{Bloom,CRW,Puppe}). The question is more interesting if one further requires that $\Diff(X)$ contains nontrivial finite groups, preferably of arbitrarily big size. If $D^n$ denotes the closed $n$-dimensional disk, then $(D^n,\Diff(D^n))$ has the fixed point property if $n\leq 4$ (this is an easy exercise for $n\leq 2$, and is much less obvious if $n$ is $3$ or $4$, see \cite{BKS}).
However, $(D^n,\Diff(D^n))$ does not have the fixed point property if $n\geq 6$: for any $n\geq 6$ there is a smooth effective action of the alternating group $A_5$ on $D^n$ (see \cite{BM}, and note that
a one-fixed-point action on $S^n$ induces a fixed point free action on $D^n$ by removing an invariant open ball in $S^n$ centered at the fixed point).
Similarly $(\RR^n,\Diff(\RR^n))$ has the fixed point property
if $n\leq 3$ and it does not have it if $n\geq 5$, see \cite{BKS}.

Given the previous results, it seems reasonable to expect that there are very few non-asymmetric manifolds $X$ such that $(X,\Diff(X))$ has the fixed point property. This motivates us to consider the following weaker notion.

We say that $(X,\gG)$ has the {\bf almost fixed point property} if there exists a constant $C$ such that for every finite subgroup
$G\leq\gG$ there exists a point $x\in X$ whose stabilizer in $G$,
say $G_x$, satisfies $[G:G_x]\leq C$.
If for a given manifold $X$ there is an upper bound on the size of the finite subgroups
of $\Diff(X)$ then $(X,\Diff(X))$ trivially has the almost fixed point property.
Our main result gives infinitely many examples of compact manifolds $X$ such that $\Diff(X)$ has
arbitrarily big finite subgroups and yet $(X,\Diff(X))$ has the almost fixed point property.

We say that a compact manifold $X$ has no odd cohomology if its integral cohomology is torsion free and
supported in even degrees. This implies that the Euler characteristic of any connected component
of $X$ is strictly positive.
Note that if $X$ is orientable and closed then the assumption that
$H^*(X;\ZZ)$ is supported in even degrees implies, by Poincar\'e duality and the universal
coefficient theorem, that the cohomology is torsion free.

\begin{theorem}
\label{thm:punts-quasi-fixos}
Let $X$ be a
smooth compact manifold, possibly with boundary, without odd
cohomology. Then $(X,\Diff(X))$ has the almost fixed point property.
More precisely, there exists a constant $C$, depending only on the dimension of $X$
and on $H^*(X;\ZZ)$, such that any finite group
$G$ acting smoothly on $X$ has a subgroup $G_0$ satisfying $[G:G_0]\leq C$ and
$\chi(Y^{G_0})=\chi(Y)$ for every connected component $Y\subseteq X$.
\end{theorem}

Theorem \ref{thm:punts-quasi-fixos} implies that $G_0$ preserves the connected components of $X$
and for any connected component $Y\subseteq X$ the fixed point set $Y^{G_0}$ is nonempty.
Note that the general statement of Theorem \ref{thm:punts-quasi-fixos} is not a formal consequence of
the particular case in which $X$ is connected.

Theorem \ref{thm:punts-quasi-fixos} implies for example that $(D^n,\Diff(D^n))$ has the almost fixed point property for every $n$.

Compactness is an essential condition in Theorem \ref{thm:punts-quasi-fixos}. Indeed,
the analogous statement for smooth finite group actions on $\RR^n$ is false for $n\geq 7$:
by a theorem of Haynes, Kwasik, Mast and Schultz \cite{HKMS},
if $n,r$ are natural numbers, $n\geq 7$ and $r$ is not a prime power,
then there exists a smooth diffeomorphism of $\RR^n$ of order $r$ without fixed points; taking
$r=pq$ with $p,q$ different primes, it follows that there is a smooth action of $G=\ZZ/r$ on
$\RR^n$ such that for every $x\in\RR^n$ the isotropy group $G_x$ is trivial or isomorphic
to $\ZZ/p$ or $\ZZ/q$; in particular $[G:G_x]\geq\min\{p,q\}$. Since $p,q$ can both be chosen to be
arbitrarily big, Theorem \ref{thm:punts-quasi-fixos} can not be true for actions on $\RR^n$.

\subsection{Algebraic actions}
The fixed point property has also been studied in algebraic geometry.
Given a real affine variety $Z$ we denote by $\Aut(Z)$ the group of algebraic automorphisms of $Z$
lifting the identity on $\Spec\RR$. Let $\AA^n$ denote the affine $n$-space over $\RR$.
A well known conjecture claims that $(\AA^n,\Aut(\AA^n))$ has the fixed point property,
see \cite{PR}. The particular case of complex affine spaces $\AA_{\CC}^n$ is perhaps more popular.
It is conjectured that any complex algebraic action of a reductive group on $\AA_{\CC}^n$ has
a fixed point, see \cite{Kraft89,Kraft}. The particular case of algebraic finite group actions
on $\AA^n_{\CC}$ is also asked in \cite{KraftSchwarz}.
There is an analogous conjecture for actions of finite $p$-groups
on affine spaces over general fields of characteristic prime to $p$ that has been popularized by Serre \cite{S3};
see \cite{Haution} for recent results and references.

The question of whether $(\AA^n,\Aut(\AA^n))$ has the fixed point property is mostly
open at present. If true, it shows a strictly algebraic property of affine
spaces, since there exist examples of real affine varieties $Z$ which are diffeomorphic to an affine space but
for which $(Z,\Aut(Z))$ does not have the fixed point property \cite{DMP}.

Using Theorem \ref{thm:punts-quasi-fixos} we obtain
the following general theorem on algebraic actions of finite groups on smooth real affine varieties,
which implies that $(\AA^n,\Aut(\AA^n))$ has the almost fixed point property for every $n$.

\begin{theorem}
Let $Z$ be a smooth real affine manifold, not necessarily compact, without odd cohomology.
Then $(Z,\Aut(Z))$ has the almost fixed point property.
\end{theorem}
\begin{proof}
If $Z$ is compact then the result follows immediately from Theorem \ref{thm:punts-quasi-fixos}.
Suppose instead that $Z$ is not compact.
Let $G\leq\Aut(Z)$ be a finite subgroup. By \cite[Theorem 4.1]{PR} there exists
a compact smooth manifold with boundary $X$, endowed with a smooth action of $G$,
such that $X\setminus\partial X$ is $G$-equivariantly diffeomorphic to $Z$.
Then $H^*(X;\ZZ)\simeq H^*(Z;\ZZ)$, so $X$ has no odd cohomology. By Theorem
\ref{thm:punts-quasi-fixos} there exists a constant $C$ such that for any finite
subgroup $\Gamma<\Diff(X)$ there exists some $\Gamma_0\leq\Gamma$ such that
$[\Gamma:\Gamma_0]\leq C$ and $\chi(X^{\Gamma_0})=\chi(X)$.
Furthermore, the constant $C$ given by Theorem \ref{thm:punts-quasi-fixos} only depends on the dimension
of $X$ and on $H^*(X;\ZZ)\simeq H^*(Z;\ZZ)$,
and hence it is independent of the compactification $X$ of $Z$.
Take $\Gamma=G$ and let $G_0\leq G$ satisfy
$[G:G_0]\leq C$ and $\chi(X^{G_0})=\chi(X)\neq 0$.
Since $X^{G_0}$ is a neat submanifold of $X$
(see Subsection \ref{ss:notation} and Lemma \ref{lemma:linearization} below),
we have $\partial(X^{G_0})=X^{G_0}\cap\partial X$, so
$\chi(X^{G_0})=\chi(X^{G_0}\setminus\partial X)$ and hence $Z^{G_0}=X^{G_0}\setminus\partial X\neq\emptyset$.
\end{proof}

The previous theorem implies that most of the cyclic group actions on $\RR^7$ constructed in \cite{HKMS}
can not be made algebraic. This also follows more directly from the existence of equivariant
compactifications and a Lefschetz fixed point formula
as in \cite{PR}.

\subsection{Jordan property and actions of abelian groups}

To prove Theorem \ref{thm:punts-quasi-fixos} we first observe,
using results on the Jordan property of diffeomorphism groups,
that it suffices to consider actions of finite abelian groups.
Recall that a group $\gG$ is said to be Jordan \cite{Po0} if there exists a constant
$C$ such that any finite subgroup $G\leq\gG$ has an abelian subgroup $A\leq G$ satisfying
$[G:A]\leq C$. The main result of \cite{M4} implies that if $X$ is a compact manifold, possibly
with boundary and without odd
cohomology, then $\Diff(X)$ is Jordan (actually we only need to assume $\chi(X)\neq 0$ for that).
Furthermore, the constant $C$ in the Jordan property for $\Diff(X)$
can be chosen to depend only on the dimension of $X$ and $H^*(X;\ZZ)$.
Hence, Theorem \ref{thm:punts-quasi-fixos} follows from the next result.

\begin{theorem}
\label{thm:main-abelian}
Let $X$ be a
compact smooth manifold (possibly with boundary) without odd
cohomology. There exists a constant $C$, depending only on the dimension of $X$
and on $H^*(X;\ZZ)$, such that any finite abelian group
$A$ acting smoothly on $X$ contains a subgroup $A_0$ satisfying
$[A:A_0]\leq C$ and $\chi(Y^{A_0})=\chi(Y)$ for every connected component $Y\subseteq X$.
\end{theorem}

To prove Theorem \ref{thm:main-abelian} we introduce a condition on finite abelian smooth
group actions on
the manifold $X$ called $\lambda$-stability, which depends on the choice of an integer $\lambda$.
We prove that
if $\lambda$ is big enough (depending on $X$) then for any $\lambda$-stable smooth action of a finite
abelian group $\Gamma$ on $X$ there exists some $\gamma\in\Gamma$ satisfying $X^{\Gamma}=X^{\gamma}$;
then the Euler characteristic of $X^{\gamma}$ can be computed using Lefschetz' formula
\cite[Exercise 6.17.3]{D}. Furthermore, for any $\lambda$ there exists a constant $C_{\lambda}$
(depending on $\lambda$ and $X$) such that for any abelian group $\Gamma$ acting smoothly on $X$
has a subgroup $\Gamma_0$ whose action on $X$ is $\lambda$-stable and
$[\Gamma:\Gamma_0]\leq C_{\lambda}$.

\subsection{The constants in Theorem \ref{thm:main-abelian}}

A natural question which we do not address here is to find the optimal values
of $C$ in Theorem \ref{thm:punts-quasi-fixos}. We also do not estimate the constants
that arise from our arguments; doing so would require in particular estimating the
constants in \cite{MT}, which plays a crucial role in \cite{M4}.
In the case of Theorem \ref{thm:main-abelian}, however,
one can give concrete (albeit
probably far from sharp) bounds. Instead of giving a general
bound, we state two theorems giving bounds on actions on disks
and even dimensional spheres. The proofs use the same ideas as
the proof of Theorem \ref{thm:main-abelian}, but restricting to
actions on disks and spheres allows us to get stronger bounds
than in the general case.

Define a map $f:\ZZ\to\NN$ as follows. For any nonnegative integer $k$ let
$$f(k)=2^k\prod_{p\geq 3}p^{[k/p]},$$
where the product is over the set of odd primes, and set
$f(k)=1$ for every negative integer $k$. Note that if $k$ is
nonnegative then $f(k)$ divides $2^{k-[k/2]}k!$.

\begin{theorem}
\label{thm:fixed-pt-disks} Let $n$ be a natural number and let
$X$ be the $n$-dimensional disk. Let $k=[(n-3)/2]$.
\begin{enumerate}
\item Any finite abelian group $A$ acting smoothly on $X$
    has a subgroup $A'\leq A$ such
     that $[A:A']$ divides
    $f(k)$, and $\chi(X^{A'})=1$.
\item For any finite abelian group $A$ acting smoothly on
    $X$ such that all prime divisors $p$ of $|A|$
    satisfy $p>\max\{2,k\}$ we have $\chi(X^{A})=1$.
\end{enumerate}
\end{theorem}

\begin{theorem}
\label{thm:fixed-pt-spheres} Let $m$ be a natural number and
let $X$ be a smooth $2m$-dimensional homology sphere.
\begin{enumerate}
\item Any finite abelian group $A$ acting smoothly on $X$
    has a subgroup $A'\leq A$ such that $[A:A']$ divides $2^{m+1}f(m-1)$,
    and $|X^{A'}|\geq 2$.
\item For any finite abelian group $A$ acting smoothly on
    $X$ such that all prime divisors $p$ of $|A|$
    satisfy $p>\max\{2,m-1\}$ we have
$|X^{A}|\geq 2$.
\end{enumerate}
\end{theorem}

\subsection{Notation and conventions}
\label{ss:notation}
We denote inclusion of sets (resp. groups) with the symbol
$\subseteq$ (resp. $\leq$); the symbol $\subset$ (resp. $<$) is reserved for {\it strict} inclusion.
If $p$ is a prime we denote by $\FF_p$ the field of $p$ elements.
When we say that a group $G$ can be generated by $d$ elements we mean that there
are elements $g_1,\dots,g_d\in G$, {\it not necessarily distinct}, which generate $G$.
All manifolds appearing in the text may, unless we say the contrary,
be open and have boundary. If a group $G$ acts on a set $X$ we denote
the stabiliser of $x\in X$ by $G_x$, and for any subset $S\subset G$ we denote
$X^S=\{x\in X\mid S\subseteq G_x\}$.

In this paper manifolds or submanifolds need not be connected and may have
connected components of different dimensions. If $X$ is a manifold and $x\in X$
we denote by $\dim_x X$ the dimension of the connected component of $X$ containing
$x$.

Following \cite[\S 1.4]{H} we say that a submanifold $Y$ of a manifold $X$ with boundary
is neat if $\partial Y=Y\cap \partial X$ and $Y$ intersects transversely the boundary $X$.

\subsection{Remark}
The main ideas in this paper are a generalization to arbitrary dimensions of those in \cite{M2},
where the Jordan property is proved for diffeomorphism groups of compact $4$-manifolds
with $\chi\neq 0$. Note that unlike \cite{M1,M4}, the arguments in \cite{M2} only
use elementary finite group theory and in particular do not rely on the CFSG.

Theorem \ref{thm:main-abelian} appeared first in the (not to be published)
preprint \cite{M1} as an ingredient of the proof
of the Jordan property for diffeomorphism groups of compact manifolds without odd cohomology.
Later it appeared in the third version of \cite{M4}, which proves that the
diffeomorphism groups of compact manifolds with $\chi\neq 0$ are Jordan. The arguments in
\cite{M4} do not use Theorem \ref{thm:main-abelian}, and Theorem \ref{thm:punts-quasi-fixos}
is deduced from combining the Jordan property with Theorem \ref{thm:main-abelian}, as we do
in the present paper. Finally we have decided to split in two parts the third version
of \cite{M4}, the first part containing the results on Jordan property (this appears in the fourth version of \cite{M4}), and the second one is the present paper.

\subsection{Contents of the paper}
Section \ref{s:preliminaries} contains some preliminary results.
In Section \ref{s:lambda-stable-actions} we introduce the notion of $\lambda$-stable action
and its basic properties. These are used in Section \ref{s:proof-thm:main-abelian}
to prove Theorem \ref{thm:main-abelian}.
Finally, in Section \ref{s:proofs-thm:fixed-pt} we prove Theorems \ref{thm:fixed-pt-disks} and \ref{thm:fixed-pt-spheres}.

\section{Preliminaries}

\label{s:preliminaries}

\subsection{Local linearization of smooth finite group actions}
\label{ss:local-linearization}
The following result is well known (see e.g. \cite{M4} for the proof). Statement (1)
implies that the fixed point set of a finite group action on a manifold with boundary
is a neat submanifold. 

\begin{lemma}
\label{lemma:linearization}
Let a finite group $\Gamma$ act smoothly on a manifold $X$, and let $x\in X^{\Gamma}$.
The tangent space $T_xX$ carries a linear action of $\Gamma$, defined as the derivative
at $x$ of the action on $X$, satisfying the following properties.
\begin{enumerate}
\item There exist neighborhoods $U\subset T_xX$ and $V\subset X$, of $0\in T_xX$ and $x\in X$ respectively, such that:
    \begin{enumerate}
    \item if $x\notin\partial X$ then there is a $\Gamma$-equivariant diffeomorphism $\phi\colon U\to V$;
    \item if $x\in\partial X$ then there is $\Gamma$-equivariant diffeomorphism $\phi\colon U\cap \{\xi\geq 0\}\to V$, where $\xi$ is a nonzero
        $\Gamma$-invariant element of $(T_xX)^*$ such that $\Ker\xi=T_x\partial X$.
    \end{enumerate}
\item If the action of $\Gamma$ is effective and $X$ is connected then the action of
$\Gamma$ on $T_xX$ is effective, so it induces an inclusion $\Gamma\hookrightarrow\GL(T_xX)$.
\item \label{item:inclusio-propia-subvarietats-fixes}
If $\Gamma'\vt \Gamma$ and $\dim_xX^{\Gamma}<\dim_xX^{\Gamma'}$ then
there exists an irreducible $\Gamma$-submodule $V\subset T_xX$
on which the action of $\Gamma$ is nontrivial but the action of $\Gamma'$ is trivial.
\end{enumerate}
\end{lemma}

\subsection{Fixed point loci of actions of abelian $p$-groups}
\label{s:fixed-point-loci-p-groups}

The following lemma is standard, see e.g. \cite[Lemma 5.1]{M7} for the proof.

\begin{lemma}
\label{lemma:cohom-no-augmenta}
Let $X$ be a manifold, let $p$ be a prime, and
let $G$ be a finite $p$-group acting continuously on $X$.
We have
$$\sum_j b_j(X^G;\FF_p)\leq \sum_j b_j(X;\FF_p).$$
In particular, the number of connected components of $X^G$
is at most $\sum_j b_j(X;\FF_p)$.
\end{lemma}

\section{$\lambda$-stable actions of abelian groups}
\label{s:lambda-stable-actions}

In this section all manifolds will be compact, possibly with boundary, and non necessarily
connected. If $X$ is a manifold we call the
dimension of $X$ (denoted by $\dim X$) the maximum of the dimensions of the connected components
of $X$.

\subsection{Preliminaries}

\begin{lemma}
\label{lemma:chains-inclusions}
Suppose that $m,k$ are non negative integers.
If $X$ is a smooth manifold of dimension $m$,
$X_1\subset X_2\subset \dots\subset  X_r\subseteq  X$
are strict inclusions of closed neat 
submanifolds, and each $X_i$ has
at most $k$ connected components, then $$r\leq \left(m+k+1\atop m+1\right).$$
\end{lemma}
\begin{proof}
Let $X_1\subset  X_2\subset \dots\subset  X_r\subseteq  X$ be as in the statement of the lemma.
For any $i$ let $d(X_i)=(d_0(X_i),\dots,d_m(X_i))$, where
$d_{j}(X_i)$ denotes the number of connected components of $X_i$ of dimension $j$.
Each $X_i$ has at most $k$ connected components, so $d(X_i)$ belongs to
$$\dD=\{(d_0,\dots,d_m)\in\ZZ_{\geq 0}^{m+1}\mid \sum d_j\leq k\}.$$
Let us prove that the map $d:\{1,\dots,r\}\to\dD$ is injective.
Suppose that $1\leq i<j\leq r$ and let $Y=X_i$ and $Z=X_j$.
Let $Y_1,\dots,Y_r$ resp. $Z_1,\dots,Z_s$ be the connected components of $Y$ resp. $Z$,
labelled in such a way that
$\dim Y_u\leq \dim Y_{u+1}$ and $\dim Z_v\leq \dim Z_{v+1}$ for every $u,v$.
Since $Y\subset Z$, there exists a map $f\colon \{1,\dots,r\}\to\{1,\dots,s\}$
such that $Y_j\subseteq Z_{f(j)}$, which implies that $\dim Y_j\leq\dim Z_{f(j)}$.
Let $J$ be the set of indices $j$ such that $\dim Y_j<\dim Z_{f(j)}$.
We distinguish two cases.

If $J=\emptyset$, so that
$\dim Y_j=\dim Z_{f(j)}$ for each $j$, then $Y\neq Z$ implies that
$Z=Y\sqcup W$ for some nonempty and possibly disconnected $W\subset X$,
because by assumption $Y$ is a closed neat submanifold of $X$.
This implies that for every $\delta$ we have
$d_{\delta}(Z)\geq d_{\delta}(Y)$, and
the inequality is strict for at least one $\delta$. Hence $d(Y)\neq d(Z)$.

Now suppose that $J\neq\emptyset$ and assume that $d(Y)=d(Z)$. We are going to
see that this leads to a contradiction. Let $\xi=\max J$.
Let $n_Y=\dim Y_{\xi}$ and $n_Z=\dim Z_{f(\xi)}$. We have $n_Y<n_Z$ because $\xi\in J$.
By assumption $Y$ has the same number of $n_Z$-dimensional connected components
as $Z$. One of the $n_Z$-dimensional connected components of $Z$ contains
an $n_Y$-dimensional connected component of $Y$. Hence, not every $n_Z$-dimensional
connected component of $Y$ is contained in an $n_Z$-dimensional connected component
of $Z$, because $Y$ is closed.
So there must exist some $Y_\alpha$ satisfying $\dim Y_\alpha=n_Z$ and $Y_\alpha\subset Z_\beta$
with $\dim Y_\alpha<\dim Z_\beta$, which implies $\alpha\in J$. But $n_Y=\dim Y_{\xi}<n_Z=\dim Y_{\alpha}$,
so $\xi<\alpha$. This contradicts the definition of $\xi$, so the proof that $d$ is injective is now complete.

The injectivity of $d$ implies that $r\leq |\dD|$. We have
$$|\dD|=\left( m \atop m\right)+\left( m+1 \atop m\right)+\dots+\left( m+k \atop m\right)
=\left( m+k+1 \atop m+1\right)$$
because $\left(m+s\atop m\right)$ is the number of $(m+1)$-tuples of nonnegative integers with total
sum $s$,
so the proof of the lemma is complete.
\end{proof}

A theorem of Mann and Su \cite[Theorem 2.5]{MS} states that for any
compact manifold $X$ there exists some integer
$$\mu(X)\in\ZZ$$
with the property that for any prime $p$ and any elementary $p$-group $(\ZZ/p)^r$ admitting
an effective action on $X$ we have $r\leq\mu(X)$.
This implies that any finite abelian $p$-group acting effectively on $X$ is
isomorphic to $\ZZ/{p^{e_1}}\oplus\dots\oplus\ZZ/{p^{e_r}}$, where $r\leq\mu(X)$ and
$e_1,\dots,e_r$ are natural numbers.

\begin{lemma}
\label{lemma:subgrups-index-petits}
Let $X$ be a compact manifold and let $n$ be a natural number. Any finite abelian
$p$-group $\Gamma$ acting effectively on $X$ has a subgroup $G\leq\Gamma$
of index $[\Gamma:G]\leq p^{n\mu(X)}$ which is contained in each subgroup $\Gamma'\leq\Gamma$
of index $[\Gamma:\Gamma']\leq p^n$.
\end{lemma}
\begin{proof}
We may assume that $\Gamma\simeq \prod_{j=1}^r\la\gamma_j\ra$, where $r\leq\mu(X)$ and
$\gamma_1,\dots,\gamma_r\in\Gamma$. We claim that
$G:=\la \gamma_1^{p^n},\dots,\gamma_r^{p^n}\ra$ has the desired property. Indeed, if
$\Gamma'\leq\Gamma$ is a subgroup satisfying $[\Gamma:\Gamma']\leq p^n$ then
the exponent of $\Gamma/\Gamma'$ divides $p^n$, and this implies that
$\gamma_j^{p^n}\in\Gamma'$ for each $j$, so $G\leq\Gamma'$. Clearly, $[\Gamma:G]\leq p^{nr}\leq p^{n\mu(X)}$.
\end{proof}

\newcommand{\tr}{\operatorname{tr}}

Suppose that a finite group $G$ acts on a compact manifold $Y$, possibly with boundary.
A $G$-good triangulation of $Y$ (see \cite[\S 2.2]{M4}) is a
pair $(\cC,\phi)$, where $\cC$ is a finite simplicial complex endowed with a
action of $G$ and $\phi\colon Y\to |\cC|$ is a $G$-equivariant homeomorphism,
and for any $g\in G$ and any $\sigma\in \cC$ such that $g(\sigma)=\sigma$
we have $g(\sigma')=\sigma'$ for any subsimplex $\sigma'\subseteq\sigma$.
The latter condition implies that for every subgroup $H\leq G$
the fixed point set $\cC^{H}$ is a simplicial
complex and $|\cC^H|=|\cC|^H$.
If both $Y$ and the action of $G$ are smooth, then there exist $G$-good triangulations
of $Y$, see \cite[\S 2.2]{M4}.

\begin{lemma}
\label{lemma:exists-Gamma-chi-p}
Let $X$ be a compact smooth manifold and let $p$ be any prime
number. There exists a natural number $C_{p,\chi}$ with the
following property. For any finite $p$-group $\Gamma$ acting smoothly on $X$ there
exists a subgroup $\Gamma_{\chi}\leq \Gamma$ of index at most $C_{p,\chi}$ satisfying
\begin{enumerate}
\item $[\Gamma:\Gamma_{\chi}]\leq C_{p,\chi}$;
\item for any subgroup $\Gamma_0\leq \Gamma_{\chi}$ we have
$\chi(X^{\Gamma_0})=\chi(X).$
\end{enumerate}
Furthermore, there exists some $P_{\chi}$ such that if
$p\geq P_{\chi}$ then $C_{p,\chi}$ can be taken to be $1$.
\end{lemma}
\begin{proof}
Let $n$ be the smallest integer such that $p^{n+1}>2\sum_jb_j(X;\FF_p)$.
We have $$|\chi(X)+ap^{n+1}|>\sum_jb_j(X;\FF_p)\qquad\text{for any nonzero integer $a$.}$$
We are going
to prove that $C_{p,\chi}:=p^{n\mu(X)}$ does the job.
Let $\Gamma$ be a $p$-group acting on $X$.
Let $\Gamma_{\tr}\leq\Gamma$ be the kernel of the morphism $\Gamma\to\Diff(X)$
given by the action. For the purposes of proving the lemma we may replace $\Gamma$ by $\Gamma/\Gamma_{\tr}$ and hence assume that $\Gamma$ acts effectively on $X$.

By Lemma \ref{lemma:subgrups-index-petits},
there exists a subgroup $\Gamma_{\chi}\leq\Gamma$ of index $[\Gamma:\Gamma_{\chi}]\leq p^{n\mu(X)}$ such that
for any subgroup $\Gamma'\leq \Gamma$ of index $[\Gamma:\Gamma']\leq p^n$ we have $\Gamma_{\chi}\leq\Gamma'$.
We now prove that any subgroup $\Gamma_0\leq \Gamma_{\chi}$ satisfies
$\chi(X^{\Gamma_0})=\chi(X)$.
Consider a $\Gamma$-good triangulation $(\cC,\phi)$ of $X$.
We have $|\cC|^{\Gamma_0}=|\cC^{\Gamma_0}|$, so
\begin{equation}
\label{eq:chi-comptar}
\chi(X)-\chi(X^{\Gamma_0})=\chi(\cC)-\chi(\cC^{\Gamma_0})=
\sum_{j\geq 0}(-1)^j\sharp\{\sigma\in \cC\setminus \cC^{\Gamma_0}\mid \dim\sigma=j\}.
\end{equation}
If $\sigma\in \cC\setminus \cC^{\Gamma_0}$ then the stabilizer $\Gamma_{\sigma}=\{\gamma\in\Gamma\mid\gamma\cdot\sigma=\sigma\}$ does not contain $\Gamma_0$.
This implies that $[\Gamma:\Gamma_{\sigma}]\geq p^{n+1}$, for otherwise
$\Gamma_{\sigma}$ would contain $\Gamma_{\chi}$ and hence also $\Gamma_0$.
Consequently, the cardinal of the orbit $\Gamma\cdot\sigma$
is divisible by $p^{n+1}$. Repeating this argument for all
$\sigma\in \cC\setminus \cC^{\Gamma_0}$ and using (\ref{eq:chi-comptar}), we conclude that $\chi(X)-\chi(X^{\Gamma_0})$ is divisible by $p^{n+1}$.

Now, we have
$|\chi(X^{\Gamma_0})|\leq \sum_j b_j(X^{\Gamma_0};\FF_p)
\leq \sum_j b_j(X;\FF_p)$
(the first inequality is obvious, and the second one follows
from Lemma \ref{lemma:cohom-no-augmenta}).
By our choice of $n$, the congruence $\chi(X^{\Gamma_0})\equiv\chi(X)\mod p^{n+1}$
and the inequality
$|\chi(X^{\Gamma_0})|\leq \sum_j b_j(X;\FF_p)$
imply that $\chi(X^{\Gamma_0})=\chi(X)$.

We now prove the last statement. Since $X$ is compact, its cohomology is finitely generated,
so in particular the torsion of its integral cohomology is bounded. Hence there exists some $p_0$
such that if $p\geq p_0$ then $b_j(X;\FF_p)=b_j(X)$ for every $j$. Define
$P_{\chi}=\max\{p_0,2\sum_j b_j(X)+1\}$. If $p\geq P_{\chi}$ then the number $n$ defined
at the beginning of the proof is equal to $0$, so $C_{p,\chi}$ can be taken to be $1$.
\end{proof}

\subsection{$\lambda$-stable actions: abelian $p$-groups}

Let $p$ be a prime and let $\Gamma$ be a finite abelian $p$-group acting smoothly
on a smooth compact manifold $X$.
Recall that for any $x\in X^{\Gamma}$ the space
$T_xX/T_x^{\Gamma}X$ (which is the fiber over $x$ of the normal
bundle of the inclusion of $X^{\Gamma}$ in $X$) carries a linear action of $\Gamma$,
induced by the derivative at $x$ of the action on $X$,
and depending up to isomorphism only on the connected component of $X^{\Gamma}$ to which
$x$ belongs.

Let $\lambda$ be a natural number.
We say that the action
of $\Gamma$ on $X$ is {\bf $\lambda$-stable} if:
\begin{enumerate}
\item $\chi(X^{\Gamma_0})=\chi(X)$
for any subgroup $\Gamma_0\leq \Gamma$;
\item for any $x\in X^{\Gamma}$ and any
character $\theta\colon \Gamma\to\CC^*$ occurring in the $\Gamma$-module
$T_xX/T_xX^{\Gamma}$ we have
$$[\Gamma:\Ker\theta]> \lambda.$$
\end{enumerate}

Note that if $\Gamma$ acts trivially on $X$ (e.g. if $\Gamma$ is trivial)
then the action is $\lambda$-stable for any $\lambda$.

When the manifold $X$ and the action of $\Gamma$ on $X$ are clear from the context,
we will sometimes abusively say that $\Gamma$ is $\lambda$-stable.
For example, if an abelian group $G$ acts
on $X$ we will say that a $p$-subgroup $\Gamma\leq G$ is $\lambda$-stable if
the restriction of the action of $G$ to $\Gamma$ is $\lambda$-stable.

\begin{lemma}
\label{lemma:exists-Gamma-stable-p}
Let $\Gamma$ be a finite abelian $p$-group acting smoothly on $X$
so that for any subgroup $\Gamma'\leq \Gamma$ we have
$\chi(X^{\Gamma'})=\chi(X)$.
If $p>\lambda$ then $\Gamma$ is $\lambda$-stable. If $p\leq \lambda$ then there exists a
$\lambda$-stable subgroup $\Gamma_{\stable}\leq\Gamma$ satisfying
$$[\Gamma:\Gamma_{\stable}]\leq\lambda^e,\qquad e=\left(m+k+1\atop m+1\right),$$
where $m=\dim X$ and $k=\sum_jb_j(X;\FF_p)$.
\end{lemma}
\begin{proof}
Suppose that $p$ is a prime number satisfying $p>\lambda$,
and that $\Gamma$ is a finite abelian $p$-group satisfying the properties in the statement of
the lemma. Then $\Gamma$ is $\lambda$-stable, because for any
$x\in X^{\Gamma}$ and any character $\theta:\Gamma\to\CC^*$ occurring in $T_xX/T_xX^{\Gamma}$
the subgroup $\Ker\theta\leq\Gamma$, being a strict subgroup (by (1) in Lemma \ref{lemma:linearization}),
satisfies $[\Gamma:\Ker\theta]\geq p>\lambda$.

Now suppose that $p$ is a prime satisfying $p\leq\lambda$
and that $\Gamma$ is a finite abelian $p$-group satisfying the properties in the statement of
the lemma. Let also $m,k,e$ be as in the statement.
We are going to prove that there exists some $\lambda$-stable subgroup $\Gamma_{\stable}\leq\Gamma$ satisfying $[\Gamma:\Gamma_{\stable}]\leq \lambda^e$.

Let $\Gamma_0=\Gamma$. If $\Gamma_0$ is $\lambda$-stable, we define $\Gamma_{\stable}:=\Gamma_0$
and we are done. If $\Gamma_0$ is not $\lambda$-stable, then
there exists some $x\in X^{\Gamma_0}$ and a
character $\theta\colon \Gamma_0\to\CC^*$ occurring in the $\Gamma_0$-module
$T_xX/T_xX^{\Gamma_0}$ such that
$[\Gamma_0:\Ker\theta]\leq\lambda$. Choose one such $x$ and $\theta$
and define $\Gamma_1:=\Ker\theta$. Then clearly
$[\Gamma_0:\Gamma_1]\leq\lambda$ and, by (1) in Lemma \ref{lemma:linearization}, we have a strict inclusion
$X^{\Gamma_0}\subset X^{\Gamma_1}$.
If $\Gamma_1$ is $\lambda$-stable, then we define
$\Gamma_{\stable}:=\Gamma_1$ and we stop, otherwise
we repeat the same procedure with $\Gamma_0$ replaced by $\Gamma_1$ and define
a subgroup $\Gamma_2\leq\Gamma_1$ satisfying $[\Gamma_1:\Gamma_2]\leq\lambda$
and $X^{\Gamma_1}\subset X^{\Gamma_2}$. And so on.
Each time we repeat this procedure, we
go from one group $\Gamma_i$ to a subgroup $\Gamma_{i+1}$ satisfying
$[\Gamma_i:\Gamma_{i+1}]\leq\lambda$ and $X^{\Gamma_i}\subset X^{\Gamma_{i+1}}$.

Suppose that we have been able to repeat the previous procedure $e$ steps, so that we
have a decreasing sequence of subgroups
$\Gamma=\Gamma_0\supset \Gamma_1\supset\dots\supset\Gamma_{e}$
giving strict inclusions
$$X^{\Gamma_0}\subset X^{\Gamma_1}\subset\dots X^{\Gamma_{e}}\subseteq X.$$
For each $j$ the manifold $X^{\Gamma_j}$ is a neat submanifold
of $X$ (by (1) in Lemma \ref{lemma:linearization}) and
the number of connected components of $X^{\Gamma_j}$ satisfies
(by Lemma \ref{lemma:cohom-no-augmenta})
$$|\pi_0(X^{\Gamma_j})|=b_0(X^{\Gamma_j};\FF_p)\leq
\sum_j b_j(X^{\Gamma_j};\FF_p)\leq \sum_j b_j(X;\FF_p)=k.$$
So our assumption leads to a contradiction with Lemma \ref{lemma:chains-inclusions}.
It follows that the previous procedure must stop before reaching the $e$-th step,
so its outcome is a sequence of subgroups
$\Gamma=\Gamma_0\supset \dots\supset\Gamma_{f}$
satisfying $[\Gamma_i:\Gamma_{i+1}]\leq\lambda$, $f<e$, and $\Gamma_{\stable}:=\Gamma_f$
is $\lambda$-stable. We also have $[\Gamma:\Gamma_{\stable}]\leq\lambda^f\leq\lambda^e$,
so the proof of the lemma is complete.
\end{proof}

\subsection{Fixed point sets and inclusions of groups}
Let $X$ be a compact manifold.
If $A,B$ are submanifolds of $X$, we will write
$$A\preccurlyeq B$$
whenever $A\subseteq B$ and each connected component of $A$ is a connected component of $B$.
Let $p$ be a prime.

\begin{lemma}
\label{lemma:subgroup-stable-general}
Let $\lambda$ be a natural number.
Let $\Gamma$ be a finite abelian $p$-group acting smoothly
on a compact manifold $X$ in a $\lambda$-stable way.
If a subgroup $\Gamma_0\leq \Gamma$ satisfies $[\Gamma:\Gamma_0]\leq\lambda$
then $X^{\Gamma}\preccurlyeq X^{\Gamma_0}$.
\end{lemma}
\begin{proof}
We clearly have $X^{\Gamma}\subseteq X^{\Gamma_0}$, so it suffices to prove that
for each $x\in X^{\Gamma}$ we have $\dim_xX^{\Gamma}=\dim_xX^{\Gamma_0}$.
If this is not the case for some $x\in X^{\Gamma}$ then, by
(\ref{item:inclusio-propia-subvarietats-fixes}) in
Lemma \ref{lemma:linearization}, there exist an irreducible $\Gamma$-submodule
of $T_xX/T_xX^{\Gamma}$ on which the action of $\Gamma_0$ is trivial.
Let $\theta\colon \Gamma\to\CC^*$
be the character associated to this submodule. Then $\Gamma_0\leq \Ker\theta$, which implies
that $[\Gamma:\Ker\theta]\leq\lambda$, contradicting the hypothesis that $\Gamma$ is $\lambda$-stable.
\end{proof}

\begin{lemma}
\label{lemma:existeix-element-generic-p}
Suppose that $\lambda\geq(\dim X)(\sum_jb_j(X;\FF_p))$, and
let $\Gamma$ be a finite abelian $p$-group acting on $X$ in a $\lambda$-stable way.
There exists an element $\gamma\in\Gamma$ such that
$X^{\Gamma}\preccurlyeq X^{\gamma}$.
\end{lemma}
\begin{proof}
Let $\nu(\Gamma)$ be the collection of subgroups of $\Gamma$ of the form
$\Ker\theta$, where $\theta\colon \Gamma\to\CC^*$ runs over the set of characters
appearing in the action of $\Gamma$ on the fibers of the normal bundle
of the inclusion of $X^{\Gamma}$ in $X$.
Since $\Gamma$ is finite, its representations are rigid, so the irreducible representations in
the action of $\Gamma$ on the normal fibers of the inclusion $X^{\Gamma}\hookrightarrow X$ are locally
constant on $X^{\Gamma}$. For each $x\in X^{\Gamma}$ the
representation of $\Gamma$ on $T_xX/T_x^{\Gamma}$ splits as the sum of at most
$\dim X$ different irreducible representations. Consequently, $\nu(\Gamma)$
has at most $\dim X |\pi_0(X^{\Gamma})|$ elements.
By Lemma \ref{lemma:cohom-no-augmenta}, $|\pi_0(X^{\Gamma})|\leq\sum_jb_j(X;\FF_p)$.
Since $\Gamma$ is $\lambda$-stable, we have $|\Gamma'|<\lambda^{-1}|\Gamma|$
for each $\Gamma'\in\nu(\Gamma)$, so
$$\left|\bigcup_{\Gamma'\in\nu(\Gamma)}\Gamma'\right|\leq \lambda^{-1}|\Gamma|
|\nu(\Gamma)|\leq \lambda^{-1}|\Gamma|\dim X\left(\sum_jb_j(X;\FF_p)\right)<|\Gamma|.$$
Consequently, there exists at least one element $\gamma\in\Gamma$ not contained in $\bigcup_{\Gamma'\in\nu(\Gamma)}\Gamma'$.
By Lemma \ref{lemma:linearization} we have $X^{\Gamma}\preccurlyeq X^{\gamma}$.
\end{proof}

\subsection{$\lambda$-stable actions: arbitrary abelian groups}
Let $X$ be a compact manifold and
let $\Gamma$ be a finite abelian group acting on $X$.
For any prime $p$ we denote by $\Gamma_{p}$ the $p$-part of $\Gamma$.
We say that the action of $\Gamma$ on $X$ is $\lambda$-stable if and only if
for any prime $p$ the restriction of the action to
$\Gamma_{p}$ is $\lambda$-stable (recall that any action of the trivial group is $\lambda$-stable).
As for $p$-groups, when the manifold $X$ and the action are clear from the context,
we will sometimes say that $\Gamma$ is $\lambda$-stable.

\begin{theorem}
\label{thm:existence-stable-subgroups}
Let $\lambda$ be a natural number.
There exists a constant $C_{\lambda}$, depending only on $X$ and $\lambda$,
such that any finite abelian group $\Gamma$ acting on $X$
has a $\lambda$-stable subgroup of index at most $C_{\lambda}$.
\end{theorem}
\begin{proof}
Let $P_{\chi}$ be the number defined in Lemma \ref{lemma:exists-Gamma-chi-p}.
Define
$$C_{\lambda}:=\left(\prod_{p\leq P_{\chi}}C_{p,\chi}\right)\left(\prod_{p\leq\lambda}\lambda^e\right),$$
where in both products $p$ runs over the set of primes satisfying the inequality and
$$e=\left( m+K+1\atop m+1\right),\qquad m=\dim X,\qquad K=\sum_j\max\{b_j(X;\FF_p)\mid p\text{ prime}\}.$$
The theorem follows from combining Lemmas \ref{lemma:exists-Gamma-chi-p} and
Lemma \ref{lemma:exists-Gamma-stable-p} applied to each of the factors
of $\Gamma\simeq\prod_{p|d}\Gamma_p$, where $d=|\Gamma|$.
\end{proof}

\subsection{$\lambda$-stable actions on manifolds without odd cohomology}
\label{s:K-stable-actions-no-odd-cohomology}

In this section $X$ denotes a
manifold without odd cohomology.
Let $p$ be any prime number.
Applying cohomology to the exact sequence of locally constant sheaves on $X$
\begin{equation*}
0\to\un{\ZZ}\stackrel{\cdot p}{\longrightarrow}\un{\ZZ}
\to \un{\FF_p}
\to 0
\end{equation*}
and using the fact that $X$ has no odd cohomology we obtain
\begin{equation}
\label{eq:betti-numbers-the-same}
b_j(X;\FF_p)=b_j(X)\quad\text{for any $j$}\qquad\Longrightarrow\qquad
\chi(X)=\sum_jb_j(X)=\sum_jb_j(X;\FF_p).
\end{equation}

\begin{lemma}
\label{lemma:no-odd-cohomology}
Let $p$ be any prime
number. Suppose that a finite abelian $p$-group $\Gamma$ acts on $X$
and that there is a subgroup $\Gamma'\leq \Gamma$ such that
$X^{\Gamma}\preccurlyeq X^{\Gamma'}$ and $\chi(X^{\Gamma})=\chi(X^{\Gamma'})=\chi(X)$.
Then $X^{\Gamma}=X^{\Gamma'}$.
\end{lemma}
\begin{proof}
Combining (\ref{eq:betti-numbers-the-same}) and Lemma
\ref{lemma:cohom-no-augmenta} we deduce:
$$\chi(X^{\Gamma})\leq \sum_jb_j(X^{\Gamma};\FF_p)\leq\sum_jb_j(X;\FF_p)=\chi(X).$$
Since $\chi(X^{\Gamma})=\chi(X)$ we have
$\sum_jb_j(X^{\Gamma};\FF_p)=\sum_jb_j(X;\FF_p).$ Applying the same arguments to $\Gamma'$
we conclude that
$\sum_jb_j(X^{\Gamma};\FF_p)=\sum_jb_j(X^{\Gamma'};\FF_p)$. Combining this
with $X^{\Gamma}\preccurlyeq X^{\Gamma'}$ we deduce $X^{\Gamma}=X^{\Gamma'}$.
\end{proof}

\begin{lemma}
\label{lemma:good-p-Gamma-has-gamma}
Let $\lambda_\chi=\chi(X)\dim X$.
If $\Gamma$ is a finite abelian $p$-group acting on $X$ in a $\lambda_\chi$-stable way,
then there exists some $\gamma\in X$ such that $X^{\Gamma}=X^{\gamma}$.
\end{lemma}
\begin{proof}
This follows from combining Lemma \ref{lemma:existeix-element-generic-p},
 equality (\ref{eq:betti-numbers-the-same}), and Lemma \ref{lemma:no-odd-cohomology}.
\end{proof}

\section{Proof of Theorem \ref{thm:main-abelian}}
\label{s:proof-thm:main-abelian}

Let $X$ be a compact manifold without odd cohomology. Let $A$ be a finite
abelian group acting
smoothly on $X$.

\newcommand{\Mat}{\operatorname{Mat}}

Let $b_j=b_j(X)$ and let $b=\sum_jb_j^2$.
We claim that there exists a subgroup $G\leq A$ whose action on
the cohomology $H^*(X;\ZZ)$ is trivial and which satisfies
$$[A:G]\leq 3^{b}.$$
To prove the claim, recall that a well known lemma of Minkowski states
that for any $n$ and any finite group $H\leq\GL(n,\ZZ)$ the restriction
of the quotient map $$q_n:\GL(n,\ZZ)\to\GL(n,\FF_3)$$ to $H$ is injective (see e.g. \cite{Mi,S1}; the proof
is easy: it suffices to check that for any nonzero $M\in\Mat_{n\times n}(\ZZ)$ and nonzero
integer $k$ the matrix $(\Id_n+3M)^k$ is different from the identity, see e.g. \cite[V.3.4]{FK}). Choosing a homogeneous basis of $H^*(X;\ZZ)$
the action of $A$ on the cohomology can be encoded in a morphism of groups
$$\phi:A\to\prod_j\GL(b_j,\ZZ).$$ Then
$$G:=\Ker (q\circ\phi),\qquad q=(q_{b_0},\dots,q_{b_n}),\qquad n=\dim X$$
has the required property, because $|\prod_j\GL(b_j,\FF_3)|\leq 3^{b}$.

Let $\lambda_{\chi}=\chi(X)\dim X$. By Theorem \ref{thm:existence-stable-subgroups}
there exists a subgroup $\Gamma\leq G$ whose action on $X$ is $\lambda_{\chi}$-stable
and which satisfies $[G:\Gamma]\leq C_{\lambda_{\chi}}$, where $C_{\lambda_{\chi}}$ depends
on $\lambda_{\chi}$ and $X$, but not on the group $G$.

There is an isomorphism $\Gamma\simeq \Gamma_{p_1}\times\dots\times \Gamma_{p_k}$,
where $p_1,\dots,p_k$ are the prime divisors of $|\Gamma|$.
Since the action of $\Gamma$ is $\lambda_\chi$-stable so is, by definition,
its restriction to each $\Gamma_{p_i}$, so by
Lemma \ref{lemma:good-p-Gamma-has-gamma} there exists, for each $i$, an element
$\gamma_i\in\Gamma_{p_i}$ such that $X^{\gamma_i}=X^{\Gamma_{p_i}}$.
Let $\gamma=\gamma_1\dots\gamma_k$. Then
$X^{\Gamma}=\bigcap_i X^{\Gamma_i}\subseteq X^{\gamma}.$
By the Chinese remainder theorem and the fact that the elements $\gamma_1,\dots,\gamma_k$
commute,
for each $i$ there exists some $e$ such that $\gamma^e=\gamma_i$.
Hence
$X^{\gamma}\subseteq X^{\gamma^e}=X^{\gamma_i}=X^{\Gamma_{p_i}}.$
Taking the intersection over all $i$ we get
$X^{\gamma}\subset \bigcap_i X^{\Gamma_i}=X^{\Gamma}.$
Combining the two inclusions we have $X^{\gamma}=X^{\Gamma}$.

Since $\gamma\in G$, the action of $\gamma$ on $X$ induces the trivial action
on $H^*(X;\ZZ)$, so in particular it preserves the connected components of $X$.
Let $Y\subseteq X$ be any connected component. Applying Lefschetz's formula
\cite[Exercise 6.17.3]{D} to the action of $\gamma$ on $Y$ we conclude that
$\chi(Y^{\gamma})=\chi(Y)$. Since $Y^{\gamma}=Y^{\Gamma}$, it follows
that $A_0:=\Gamma$ has the desired properties. Furthermore,
$$[A:A_0]\leq 3^{b}C_{\lambda_{\chi}}.$$
In the course of our arguments no information on $X$ apart from its
dimension and $H^*(X;\ZZ)$ has been used, so the proof of the theorem is now complete.

\section{Proofs of Theorems \ref{thm:fixed-pt-disks} and \ref{thm:fixed-pt-spheres}}
\label{s:proofs-thm:fixed-pt}

\subsection{Proof of Theorem \ref{thm:fixed-pt-disks}}

The following well known fact immediately proves the theorem in
the cases $n=1,2$.

\begin{lemma}
\label{lemma:fixed-point-disk-low-dim} Let $n$ be either $1$ or
$2$, and let $X$ be the $n$-dimensional disk. The fixed point
locus of any smooth action of a finite group on $X$ is
contractible.
\end{lemma}

Let now $n\geq 3$ be a natural number and let $X$ be the
$n$-dimensional disk.
Let $p$ be a prime and suppose that a finite abelian $p$-group
$A$ acts on $X$. Smith theory implies that the fixed point
set $X^{A}$ is $\FF_p$-acyclic (see \cite[Corollary III.4.6]{B} for the case
$A=\ZZ/p$ and use induction on $|A|$ for the general case, as
in the proof of Lemma \ref{lemma:cohom-no-augmenta} given in \cite{M7}).
In particular, $X^{A}$ is nonempty and connected.

\begin{lemma}
\label{lemma:primer-p-disk} Let $x\in X^{A}$ be any point,
and let $\theta_1,\dots,\theta_r$
be the different (real)
irreducible representations of $A$ appearing in
$T_xX/T_xX^{A}$. Let $l=\dim X^{A}$.
\begin{enumerate}
\item If $p=2$ then $r\leq n-l$. If $p$ is odd then $n-l$
    is even and $r\leq (n-l)/2$.
\item There exists some $\gamma\in A$ and a subgroup
    $A'\leq A$ such that
    $X^{\gamma}=X^{A'}$ and $[A:A']$ divides
    $p^{[r/p]}$.
\end{enumerate}
\end{lemma}
\begin{proof}
(1) is clear. To prove (2), define $A_j:=\Ker\theta_j$,
let $e_j=\log_p[A:A_j]$, and consider the function
$I:A\to\ZZ$ defined as $I(\gamma)=\sum_{j\mid
\gamma\in A_j}e_j.$ Since each $\theta_j$ is nontrivial
(e.g. by (2) in Lemma \ref{lemma:linearization}) we have
$e_j\geq 1$ for every $j$. Now we estimate
$$\sum_{\gamma\in A}I(\gamma)=\sum_{\gamma\in A}\sum_{j\mid\gamma\in A_j}e_j=
\sum_{j=1}^r|A_j|e_j=\sum_{j=1}^r\frac{|A|}{p^{e_j}}e_j\leq
\sum_{j=1}^r\frac{|A|}{p}=\frac{r}{p}|A|,$$ where the
inequality follows from the fact that the function $\NN\ni
n\mapsto n/q^n$ is non increasing for any integer $q\geq 2$.
Hence the average value of $I$ is not bigger than $r/p$, so
there exists some $\gamma\in A$, which we fix for the rest
of the argument, such that $I(\gamma)\leq [r/p]$. Let
$A':=\bigcap_{j\mid\gamma\in A_j}A_j$.

We claim that $X^{\gamma}=X^{A'}$. Since
$\gamma\in A'$, the inclusion $X^{A'}\subseteq
X^{\gamma}$ is clear. To prove the reverse inclusion observe
that, by Smith theory, both $X^{A'}$ and $X^{\gamma}$ are
acyclic, hence connected, so it suffices to prove that
$T_xX^{\gamma}\subseteq T_xX^{A}$. Suppose that $\theta_j$ maps $A$ to $\GL(W_j)$.
Let $T_xX/T_xX^{A}=V_1\oplus\dots\oplus V_r$ be the
decomposition in isotypical real representations of $A$,
where $V_j$ is isomorphic to the direct sum of a number of copies of $W_j$.
Since we clearly have $$T_xX^{A}\oplus\bigoplus_{j\mid
\gamma\in A_j}V_j\subseteq T_xX^{A'},$$ it suffices to
prove that
$$T_xX^{\gamma}\subseteq
T_xX^{A}\oplus\bigoplus_{j\mid \gamma\in A_j}V_j.$$
The latter is equivalent to proving that $T_xX^{\gamma}\cap
V_i=\Ker(\theta_i^V(\gamma)-\Id)=\{0\}$ for every $i$ such that
$\gamma\notin A_i$ (here $\theta_i^V:A\to\GL(V_j)$ is
given by restricting the action of $A$ on
$T_xX/T_xX^{A}$ to $V_j$). Since $V_i$ is isotypical and $\gamma$ is
central in $A$, $\Ker(\theta_i^V(\gamma)-\Id)$ is either
$\{0\}$ or $V_i$. But $\Ker(\theta_i^V(\gamma)-\Id)=V_i$ would
imply $\gamma\in\Ker\theta_i$, contradicting the choice of $i$.
Hence $\Ker(\theta_i^V(\gamma)-\Id)=\{0\}$ and the proof that
$X^{\gamma}=X^{A'}$ is complete.

To finish the proof of the lemma we observe that
$[A:A']$ divides $\prod_{j\mid
\gamma\in A_j}[A:A_j]=p^{I(\gamma)}$. Since
$I(\gamma)\leq [r/p]$, we deduce that $[A:A']$
divides $p^{[r/p]}$.
\end{proof}

Now let $A$ be a finite abelian group acting on $X$. For
each prime $p$ let $A_p\leq A$ denote the
$p$-part of $A$, so that $A=\prod_pA_p$. If for
some $p$ we have $\dim X^{A_p}\leq 2$, then the
classification of manifolds (with boundary) of dimension at
most $2$ implies that $\dim X^{A_p}$ is a disk, because
$\dim X^{A_p}$ is $\FF_p$-acyclic; so applying Lemma
\ref{lemma:fixed-point-disk-low-dim} to the action of $A$
on $X^{A_p}$ we deduce that
$\chi(X^{A})=\chi((X^{A_p})^{A})=1$ and the
proof of the theorem is complete.

Hence it suffices to consider the case when $\dim
X^{A_p}\geq 3$ for each $p$. Let $k=[(n-3)/2]$. Applying
Lemma \ref{lemma:primer-p-disk} for each prime $p$ we deduce
that there exists some subgroup $A_2'\leq A_2$
satisfying $[A_2:A_2']\leq 2^k$ and an element
$\gamma_2\in A_2'$ such that $X^{\gamma_2}=X^{A_2'}$
and, for each odd prime $p$, there exists some subgroup
$A_p'\leq A_p$ satisfying
$[A_p:A_p']\leq p^{[k/p]}$ and an element
$\gamma_p\in A_p'$ such that $X^{\gamma_p}=X^{A_p'}$.
Let $A'=\prod A_p'$ and let $\gamma=\prod\gamma_p$.
Arguing as in Section \ref{s:proof-thm:main-abelian} we prove
that $X^{\gamma}=X^{A'}$. Clearly $[A:A']$
divides $f(k)$, so statement (1) of Theorem
\ref{thm:fixed-pt-disks} is proved. Statement (2) follows
immediately from statement (1), because none of the odd prime
divisors of $f(k)$ is bigger than $k$.

\subsection{Proof of Theorem \ref{thm:fixed-pt-spheres}}
We follow a scheme similar to the proof of Theorem
\ref{thm:fixed-pt-disks}. Let $m$ be a natural number, let $p$
be a prime and let $Y$ be a $\FF_p$-homology $2m$-sphere. For
any smooth action of $\ZZ/p$ on $X$ the fixed point set is a
$\FF_p$-homology $s$-sphere \cite[IV.4.3]{B}. Furthermore, if
$p$ is odd, the difference $2m-s$ is even \cite[IV.4.4]{B}.
Hence we may apply the same inductive scheme as in Lemma
\ref{lemma:cohom-no-augmenta} (or \cite[IV.4.5]{B}) and deduce
the following.

\begin{lemma}
\label{lemma:sphere-p-odd} For any odd prime $p$, any finite
$p$-group $A$, and any action of $A$ on a
$\FF_p$-homology even dimensional sphere the fixed point
set is a $\FF_p$-homology even dimensional sphere
(in particular, the fixed point set is nonempty).
\end{lemma}

The case $p=2$ works differently.
Suppose that $Y$ is a smooth $n$-dimensional manifold and that
$Y$ is a $\FF_2$-homology $n$-sphere. Suppose that $\ZZ/2$
acts smoothly on $Y$. Then $Y^{\ZZ/2}$ is a $\FF_2$-homology $s$-sphere
\cite[IV.4.3]{B}. The fixed point set $Y^{\ZZ/2}$ is also a smooth submanifold
of $Y$ and $s$ coincides with the dimension of $Y^{\ZZ/2}$ as a manifold.
The condition that $Y$ is a $\FF_2$-homology sphere implies that $Y$ is compact and orientable.
One checks, using Lefschetz' formula
\cite[Exercise 6.17.3]{D} and arguing in terms of volume forms,
that $n-s$ is even if and only if the action of the
nontrivial element of $\ZZ/2$ on $Y$ is orientation preserving
(this works more generally for continuous actions on finite Hausdorff spaces
whose integral homology is finitely generated and whose $\FF_2$-homology
is isomorphic to $H_*(S^n;\FF_2)$,
by a theorem of Liao \cite{L}, see also \cite[IV.4.7]{B}).

\begin{lemma}
\label{lemma:sphere-2} For any finite $2$-group $A$ and
any smooth action of $A$ on a smooth
$\FF_2$-homology $2m$-sphere ($m\in\ZZ_{\geq 0}$) there exists a
subgroup $A_0\leq A$ whose index
$[A:A_0]$ divides $2^{m+1}$ and whose fixed point set
$X^{A_0}$ is a smooth $\FF_2$-homology even dimensional sphere
(in particular, $X^{A_0}$ is nonempty).
\end{lemma}
\begin{proof}
We use ascending induction on $|A|$.
The case $|A|=2$ being obvious, suppose that $|A|>2$ and that
the lemma is true for $2$-groups with less elements than $A$.
Suppose that $A$ acts smoothly on a compact smooth $\FF_2$-homology
$2r$-sphere $Y$. If the action of $A$ is not effective, then it factors
through a quotient of $A$, and applying the inductive hypothesis the lemma
follows. So assume that the action of $A$ on $Y$ is effective.
Let $A'\leq A$ be the subgroup consisting
of those elements whose action is orientation preserving. Then $[A:A']$
divides $2$. Let $A''\leq A'$ be a central subgroup isomorphic
to $\ZZ/2$. Then $Y^{A''}$ is a compact smooth $\FF_2$-homology even dimensional
sphere satisfying $\dim Y^{A''}\leq 2r-2$. Furthermore, $A'/A''$ acts smoothly on
$Y^{A''}$. To finish the proof, apply the inductive
hypothesis to this action.
\end{proof}

Let now $X$ be a smooth homology $2r$-sphere and suppose that a finite
abelian group $A$ acts smoothly on $X$. For any prime $p$ let $A_p\leq A$
denote the $p$-part.
By Lemma \ref{lemma:sphere-p-odd},
for any odd prime $p$ the fixed point set $X^{A_p}$ is an even dimensional $\FF_p$-homology
sphere.
By Lemma \ref{lemma:sphere-2}, $A_2$ has a subgroup $A_{2,0}$ whose index
divides $2^{r+1}$ and whose fixed point set is an even dimensional $\FF_2$-homology sphere
(in particular, it is nonempty). Replace $A_2$ by $A_{2,0}$ and define $A_0:=\prod_pA_p$,
so that $[A:A_0]$ divides $2^{r+1}$.

Suppose that for some prime $p$ the fixed point set $X^{A_p}$ is $0$-dimensional. Then $X^{A_p}$ consists of two points, $A$ acts on $X^{A_p}$, and there is a subgroup
$A'\leq A$ whose index divides $2$ and whose action on $X^{A_p}$ is trivial.
It follows that $|X^{A'}|\geq 2$ and we are done.

Suppose now that for each prime $p$ the fixed point set $X^{A_p}$ is an even dimensional
$\FF_p$-homology sphere of dimension at least $2$. In particular, $X^{A_p}$ is nonempty and
connected for every
$p$, and so is $X^{A_p'}$ for every subgroup $A_p'\leq A_p$ (since $X^{A_p'}$ is a
$\FF_p$-homology sphere and, given the inclusion $X^{A_p}\subseteq X^{A_p'}$, the dimension
of $X^{A_p'}$ is at least $2$). This property allows to use the same arguments in the proof
of Lemma \ref{lemma:primer-p-disk} to prove the following.

\begin{lemma}
\label{lemma:primer-p-esfera} Let $p$ be any prime, let
$x\in X^{A_p}$ be any point,
and let $\theta_1,\dots,\theta_r$
be the different real
irreducible representations of $A_p$ appearing in
$T_xX/T_xX^{A_p}$. Let $2l=\dim X^{A_p}$.
\begin{enumerate}
\item If $p=2$ then $r\leq 2m-2l$. If $p$ is odd then $r\leq m-l$.
\item There exists some $\gamma\in A_p$ and a subgroup
    $A_p'\leq A_p$ such that
    $X^{\gamma}=X^{A_p'}$ and $[A_p:A_p']$ divides
    $p^{[r/p]}$.
\end{enumerate}
\end{lemma}

By Lemma \ref{lemma:primer-p-esfera}
there exists some subgroup $A_2'\leq A_2$
satisfying $[A_2:A_2']\leq 2^{2m-2}$ and an element
$\gamma_2\in A_2'$ such that $X^{\gamma_2}=X^{A_2'}$
and, for each odd prime $p$, there exists some subgroup
$A_p'\leq A_p$ satisfying
$[A_p:A_p']\leq p^{[m-1/p]}$ and an element
$\gamma_p\in A_p'$ such that $X^{\gamma_p}=X^{A_p'}$.
Let $A'=\prod A_p'$ and let $\gamma=\prod\gamma_p$.
As in Section \ref{s:proof-thm:main-abelian} we have
$X^{\gamma}=X^{A'}$. The index $[A:A']$
divides $2^{m+1}f(m-1)$, so statement (1) of Theorem
\ref{thm:fixed-pt-spheres} is proved. Statement (2) follows
immediately from statement (1), because none of the odd prime
divisors of $f(m-1)$ is bigger than $m-1$.

\end{document}